\documentclass[11pt]{article}      
\usepackage{amsmath,amssymb,amsfonts} 
\usepackage{graphics}                 
\usepackage{color}                    
\usepackage{hyperref}                 
\usepackage{amsthm}
\usepackage[numbers]{natbib}\citeindextrue
\usepackage[pdftex]{graphicx}
\usepackage{comment}
\usepackage[left=1in,top=1in,bottom=1in,right=1in]{geometry}
\usepackage[algoruled,lined]{algorithm2e}
\usepackage{algorithmic}
\usepackage{multirow}

\newcommand{\R}{\mathbb{R}}
\newcommand{\p}{\mathbb{P}}
\newcommand{\pp}{\mathbf{p}}
\newcommand{\E}{\mathbb{E}}
\newcommand{\1}{\mathbf{1}}
\newcommand{\T}{\mathcal{T}}

\newcommand{\G}{\mathbb{G}}

\newcommand{\g}{\, | \,}

\newtheorem{theorem}{Theorem}[section]

\newtheorem{proposition}[theorem]{Proposition}

\begin{document}

\title{Stochastic Search with an Observable State Variable
}
\author{Lauren A. Hannah\thanks{Department of Operations Research and Financial Engineering, Princeton University, Princeton, NJ, 08544, lhannah@princeton.edu} \and Warren B. Powell\thanks{Department of Operations Research and Financial Engineering, Princeton University, Princeton, NJ, 08544} \and David M. Blei\thanks{Department of Computer Science, Princeton University, Princeton, NJ 08544}}
\maketitle

\begin{abstract}
In this paper we study convex stochastic search problems where a noisy objective function value is observed after a decision is made.  There are many stochastic search problems whose behavior depends on an exogenous state variable which affects the shape of the objective function.  Currently, there is no general purpose algorithm to solve this class of problems.  We use nonparametric density estimation to take observations from the joint state-outcome distribution and use them to infer the optimal decision for a given query state.  We propose two solution methods that depend on the problem characteristics: function-based and gradient-based optimization.  We examine two weighting schemes, kernel-based weights and Dirichlet process-based weights, for use with the solution methods.  The weights and solution methods are tested on a synthetic multi-product newsvendor problem and the hour-ahead wind commitment problem.  Our results show that in some cases Dirichlet process weights offer substantial benefits over kernel based weights and more generally that nonparametric estimation methods provide good solutions to otherwise intractable problems.
\end{abstract}


\section{Introduction}\label{sec:intro}
Stochastic search is a class of stochastic optimization problems where we have to find a deterministic parameter to minimize the expectation of a function of uncertain quantities.  The expectation is usually hard to compute, requiring instead the use of Monte Carlo samples.  The problem is typically written
\begin{equation}\label{eq:so}
\min_{x\in \mathcal{X}} \, \E\left[ F(x,Z)\right],
\end{equation}where $x\in \R^d$ is the decision, $\mathcal{X}$ is a decision set, $Z : \Omega \rightarrow \Psi$ is a random outcome and $F:\R^d \times \Psi \rightarrow \R$ is an objective function.  A classic example is the newsvendor problem where we have to stock a quantity of product to serve an uncertain demand with an unknown distribution. Each iteration, we can only observe how much we sold, after which we make adjustments. A rich theory has evolved to address problems of this type (see \citep{Sp03}).

In this paper, we introduce an important variation of the stochastic search problem.  Assume that we are first allowed to observe a state variable $S$ (such as whether it is raining or sunny) which changes our belief about the distribution of the random vector $Z$.  After observing $S$, we then choose $x$, and only then do we observe $Z$, or we may only observe $F(x,S,Z)$ (or its derivative). Each iteration starts with a new state, after which we choose an action and then observe the results.  Since information from the current state, decision and observation is used to update beliefs for future decisions, we have two challenges: 1) assembling information from previous state-decision-outcome pairs into something that can be used to make a decision for the current state and 2) finding the optimal decision given a state.

If the state space is small (say, rainy or sunny)  we can use classical methods from stochastic search by simply conditioning on the state when we do our updates. But it is often the case that $S$ is a vector, frequently with continuous elements.  The hour ahead wind commitment problem is an example:  a wind farm manager must pledge how much energy she will provide to a utility company an hour in the future.  If too much energy is pledged, the difference must be bought; if too little is pledged, the difference is lost.  The objective function depends on the future wind and market price, both unknown.  The last 24 hours of wind and market prices, time of day and time of year all contain information about the objective function.  This problem cannot be solved using standard techniques from stochastic search and stochastic optimization. 

To combat the first problem, sharing information across observations, we propose using nonparametric density estimation for the joint state and outcome distribution to group observations from ``similar'' states with weights.  To combat the second problem, making a decision given an observed state, we use the weighted observations to construct convex, deterministic approximations of the conditional expectation of the objective function.  Care is taken to ensure that the resulting optimization problems are computationally feasible.

With this high level summary in mind, we turn to a more formal description of the problem setting.  When we include the state variable, the stochastic search problem of Equation (\ref{eq:so}) becomes
\begin{equation}\label{eq:soState}
\min_{x\in\mathcal{X}} \E\left[F(x,s,Z)|S=s\right].
\end{equation}Note that the function $F$ itself may change with the state.  Conventional stochastic search techniques require us to sample from the conditional distribution $p(Z|S=s)$, treating each state observation independently~\citep{Sp03}.  We use nonparametric density estimation for the joint distribution of $(S,Z)$ to weight the states because similar values of $S$ usually affect $Z$ and $F$ in a similar way.

We propose a new model-free method to solve the stochastic
optimization problem with an observable state variable.  Our problem is motivated by online applications where we are given a state, and then after making a decision, we are given the realization of $Z$ which depends on the state.  For this reason, we index estimates and random
variables, such as $S_n$, with a subscript that indicates at which
iteration the value can be used.  We use a nonparametric density
estimate of $(S,Z)$ to weight previous observations $(S_i,
Z_i)_{i=0}^{n-1}$.  Given an observed state, we generate an estimate of $\E[F(x,s,Z)]$, called the
approximate function $\bar{F}_n(x|s)$, based on the weighted previous observations.  For this paper, we are
concerned with two classes of stochastic search solution methods
for convex problems:
\begin{itemize}
\item{\bf Function-Based Optimization:} given a state $S_n=s$ and an outcome $Z(\omega_{n+1})$ with $\omega_{n+1}\in \Omega$, the entire response function $F(x,s,Z(\omega_{n+1}))$ is known.
\item{\bf Gradient-Based Optimization:} given a state $S_n=s$, a decision $x_N$ and an outcome $\omega_{n+1}$, we only observe the stochastic gradient $\hat{\beta}(x_n,s,Z(\omega_{n+1})) =  \nabla_x F(x_n,s,Z(\omega_{n+1})).$
\end{itemize}
The wind commitment problem can be solved by function-based
optimization; once the wind speed at time $n+1$ is known, the value
for all possible commitment levels at time $n$ is also known.  In many
problems the entire objective function is too expensive too compute or
cannot be explicitly computed; however, it is often possible to
observe or estimate derivatives around a decision value.  For example,
many resource allocation problems involve solving a linear program;
the dual variables provide the gradients.  For more complicated
problems, function-based optimization may produce
functions too complicated for use in a solver.  Gradient-based
optimization eliminates this problem by restricting the form of the
approximate function $\bar{F}_n(x|s)$ to piecewise linear, separable
and convex.  In both methods, however, we use weights derived from a
joint distribution of $(S,Z)$ with previous observations to form an
approximate function, $\bar{F}_n(x|s)$.

A large class of function-based methods currently exist for problems without a state variable~\citep{Ro96,ShWa96,ShHoKi02}.  We craft a function-based optimization algorithm by extending the existing methods to weighted observations.  We give conditions for almost sure convergence of this algorithm to a global optimum.

Our approach to gradient-based optimization is less straightforward.  Stochastic approximation~\citep{RoMo51,KiWo52} is the most popular gradient search method for problems without a state variable, but for reasons explained in Section \ref{sec:sa}, it cannot be extended to problems with a state variable.  Instead, we try to construct a function that has the same behavior as the original objective function around the optimum.  Like \citep{PoRuTo04}, we use a separable, convex, piecewise linear approximation to do this, except that our previous observations are weighted according to the current state so that they produce appropriate slopes for the piecewise linear approximation.  
We show that the resulting algorithm converges to an arbitrarily small neighborhood around the optimum with probability one when the true objective function is itself separable.

Both methods rely heavily on weighting functions.  We give two methods to generate weights:  kernels and Dirichlet process mixture models.  Kernels are easy to implement and often give good results.  They can develop problems, however, when the state variable is moderate to high dimensional by giving all but a few observations weights that are effectively zero.  This can lead to unstable results.  As an alternative in these situations, we propose using weights generated by Dirichlet process mixture models.  Dirichlet process mixture models are Bayesian nonparametric models that produce a distribution over data partitions.  In effect, they cluster data in a Bayesian manner.  We derive weights from this model by placing equal weight on all previous observations that are in the same cluster as the current observation.  Then we approximate the average of these weights by taking a Monte Carlo sampling of clusterings.  This method requires more work, but it is far more stable than kernel methods.  We give conditions for when kernel and Dirichlet process weights satisfy the convergence criteria for both optimization algorithms. 

We test our methods on two problems, a two-product newsvendor problem and the hour ahead wind commitment problem.  In the two-product newsvendor problem, we use synthetic data and compare both optimization methods under each weighting function.  
In the hour ahead wind commitment problem, we use synthetic price data and wind data from the North American Land Data Assimilation System.  Due to the computational difficulties of computing weights and testing solutions every iteration, we only compare function-based optimization under the different weighting schemes.  Dirichlet process weights produce better results for this problem.

We contribute novel algorithms to include state variables in function-based optimization and gradient-based optimization problems.  We study two methods to do this:  kernel weights and Dirichlet process weights.  This is a new use of Dirichlet process mixture models.  We give empirical analysis for these methods where we show promising results on test problems.  

The paper is organized as follows.  In Section \ref{sec:literature}, we review the treatment of search and optimization problems in the presence of a state variable in different communities.  In Section \ref{sec:samplePath}, we review established function-based optimization methods, propose an algorithm that incorporates a state variable and prove convergence of that algorithm.  In Section \ref{sec:sa}, we review current gradient-based optimization methods, propose an algorithm that incorporates a state variable and prove convergence of the algorithm under certain conditions.  In Section \ref{sec:DP}, we present two weighting schemes, kernel and Dirichlet based weights.  We present an empirical analysis of our methods for synthetic newsvendor data and the hour ahead wind commitment problem in Section \ref{sec:numbers} and a discussion in Section \ref{sec:discussion}.

\section{Literature review}\label{sec:literature}
Several communities, including operations research, optimization and machine learning, have studied problems with forms similar to the stochastic search problem with a state variable of Equation (\ref{eq:soState}).  The problems and solution methods are diverse; even within communities, optimization problems with a state variable are never treated as an entire problem class.  We briefly outline the resulting hodgepodge of problems and methods.

Most commonly, stochastic search problems with a state variable are considered individually rather than as an entire problem class.  The prevailing approach is to construct a model for $Z$ and use the information in $S$ to supply the parameters.  In the case of wind farms, \citep{SoHaRo02} constructs a model for the wind and use a state of the world to determine parameters.  They also select which state variables are needed to parameterize the wind distribution.  Collectively, creating a model and selecting state variables to generate parameters for that model can require substantial time and domain knowledge.

In some areas, problems with state variables have become their own classes.  In the statistics and operations research community, there has been some study of portfolio and bandit problems in the presence of a state variable (called a ``covariate'' in the bandit literature and ``side information'' in the learning theory community).  \citep{CoOr96} and \citep{HeScSi98} studied portfolio optimization with a finite state variable, but handled it in a manner that amounts to treating each value as a separate problem.

Bandit problems with a state variable are another established area of study.  A bandit problem is a sequential decision problem with small, finite set of statistical populations (arms).  At each iteration, only one arm can be sampled; a random reward $R_i$ is obtained with probability $p_i$, where $i$ denotes the arm number.  When a state variable $S=s$ is included, the probability of a reward is $p_i(s)$.  The goal is to maximize the average reward.  This bandit variant was first introduced by \citep{Wo79} and has been studied when the distributions are assumed to have a parametric form~\citep{Sa91,WaKuPo05,GoZe09}.  They have also been studied where the mean function has been estimated in a nonparametric fashion~\citep{YaZh02,RiZe10}; however, decision-making mechanisms vary widely due to the non-convex decision set.

In the optimization community, parametric nonlinear programming includes what can be viewed as a state variable in a math programming setup.  The basic problem has the form
\begin{equation}\notag
\min_{x\in \mathcal{X}} F(x,s),
\end{equation}where $s$ ``parameterizes'' the program.  Such problems have been used for sensitivity analysis~\citep{JoWe90,RaDe95} and have been the focus of renewed interest in the model predictive control community~\citep{BeFi06}.  Parametric nonlinear programming, however, is deterministic and it assumes that $F$ is known for a given $s$.

The machine learning community solves problems with state variables more than any other community.  Machine learning is a catch-all term for a large set of subfields, including learning theory, reinforcement learning, classification, Bayesian nonparametrics and many others.  Problems with a state variable arise in some of these subfields, such as reinforcement learning and learning theory.  

State variables arise in a general way in dynamic programming and reinforcement
learning. In these problems, we might be in state $S$ and take an
action $a$, which determines, or influences, the next state
$S^{\prime}$ that we visit. The choice of the action $a$, then, needs to
consider the expected value of being in state $S^{\prime}$. A host of
algorithms have been proposed to solve this problem class (see, for
example, \citep{Pu94,BeTs96,SuBa98,TsVa01,Po07,Be07}), where the major
complication is that we do not know, and have to approximate, the
value of being in state $S^{\prime}$.  Our problem class is closest to
the field of stochastic search~\citep{Sp03}, where we often have to
face the challenge of vector-valued (and possibly continuous) decision
vectors $x$.  Of all the algorithms in reinforcement learning, our
problem is most similar to $Q$-learning which requires estimating the
value of $Q(S,a)$ (in our notation, $Q(S,x)$).  Although we do not
have to deal with the value of the downstream state, we do have to
address the challenge of complex state variables and vector-valued
decisions, which the $Q$-learning literature has not addressed.

State variables are also common in learning theory, where they are called ``side information.''  The portfolio references and many of the bandit references are written from a learning theory perspective.  The work closest to ours is \citep{HaMe07}; they incorporated a state variable into an online convex optimization problem.  In this setting, each iteration a player selects a decision from a convex set and an adversary selects a loss function from a finite set of options.  The state variable contains some information about the set of loss functions.  They construct an algorithm that minimizes regret, a notion of loss under a worst-case scenario, rather than expected loss.  They propose using a combination of a nearest-neighbor and $\epsilon$-net mapping from the state space to the decision space under smoothness constraints.  Values are updated by gradient observations.  Their algorithm does not converge to a fixed decision for a given state with more than one possible loss function.

We now turn to the first of our methods, function-based optimization.

\section{Function-based optimization with an observable state variable}\label{sec:samplePath}
We use function-based optimization when a single outcome $\omega$ can tell us the value of all decisions given that outcome~\citep{HeSc91,Ro96,ShWa96,ShHoKi02}.  For example, in the hour ahead wind commitment problem, if the wind is known, then the value of all commitment levels is known.  Function-based optimization relies on sampling a set of scenarios, $\omega_1,\dots,\omega_n$ from $\Omega$, to approximate the expectation:
\begin{equation}\label{eq:spo}
\min_{x \in \mathcal{X}} \frac{1}{n} \sum_{i=1}^n F(x,Z(\omega_i)).
\end{equation}
Equation (\ref{eq:spo}) is deterministic and deterministic methods can be used.  They can accommodate complex constraint sets but require that the entire function is known for an outcome $\omega$.  The following asymptotic results hold for function-based optimization.  Let $x^*$ be the true solution and $x^*_n$ be the solution to Equation (\ref{eq:spo}).  Under sufficient conditions, $x^*_n \rightarrow x^*$ almost surely as $n \rightarrow\infty$~\citep{Ro96,PlFuRo96,GuYoRo99}.  Moreover, \citep{Sh91} and \citep{RuMe98} show asymptotic normality under stricter conditions than those required for convergence.

Function-based optimization has been well studied, but under a variety of names.  \citep{Ro96}, \citep{ShWa96}, and \citep{GuYoRo99} call it sample path optimization; \citep{RuSh93} use likelihood ratios to approximate the optimization problem, calling it the stochastic counterpart method; \citep{HeSc91} studied this method and call it retrospective optimization; within the stochastic programming literature, it is often known as sample average approximation~\citep{ShHoKi02,KlShHo02} and scenario optimization~\citep{BiLo97}.  A similar method has been used in discrete event systems and gradient estimation, called perturbation analysis~\citep{Ho85,Gl90}.

\subsection{Algorithm for function-based optimization}
Extensions of existing algorithms to include a state variable are fairly straightforward.  New observations of $F$ have the form $F(x,S_i,Z(\omega_{i+1}))$ where $S_i$ is a random state variable.  It is difficult to place a prior over the space of all convex functions or create a joint distribution over states and the space of convex functions.  Therefore, we develop a locally weighted average of the observations to approximate the true mean function at a given state.

Let $s\in\mathcal{S}$ be a fixed query state.  We would like to minimize $\E[F(x,s,Z)]$ with respect to $x$.  Let $(S_i,Z(\omega_{i+1}))_{i=0}^{n-1}$ be a set of $n$ observations and $(w_n(s,S_i))_{i=0}^{n-1}$ be a set of weights based on the query state and the observed states where $$\sum_{i=0}^{n-1} w_n(s,S_i) = 1.$$  The weight functions may change with the number of observations $n$.  Set
\begin{equation}\label{eq:spoState1}
\bar{F}_n(x|s) = \sum_{i=0}^{n-1} w_n(s,S_i) F(x,S_i,Z(\omega_{i+1})).
\end{equation}The optimization problem becomes
\begin{equation}\label{eq:spoState2}
\min_{x\in \mathcal{X}} \, \bar{F}_n(x|s).
\end{equation}This produces an average of observations weighted by how close the observed states are to the current state.  $F(x,S_i,Z(\omega_{i+1}))$ is convex in $x$ for every $S_i$ and $\omega_{i+1}$, so $\bar{F}_n(x|s)$ is convex.  This is particularly helpful because Equation (\ref{eq:spoState2}) can then be solved with any number of deterministic solvers.  We give the general procedure in Algorithm \ref{alg:function}. 

\begin{algorithm}[h]\caption{Function-based optimization with an observable state variable}
  \begin{algorithmic}[1]\label{alg:function}
  \REQUIRE Query state $s$.
  \FOR{$i=0$ to $n-1$}
  \STATE Observe random state $S_i$.
  \STATE Observe random function $F(x,S_i,Z(\omega_{i+1}))$.
  \ENDFOR
  \STATE Generate weights $(w_n(s,S_i))_{i=0}^{n-1}$.
  \STATE Set $$\bar{F}_n(x|s) = \sum_{i=0}^{n-1} w_n(s,S_i) F(x,S_i,Z(\omega_{i+1})).$$
  \STATE Solve $$x^*_n(s) = \min_{x \in \mathcal{X}} \bar{F}_n(x|s).$$
  \end{algorithmic}
\end{algorithm}

Define $x^*_n(s)$ as $$x^*_n(s) = \arg \min_{x\in \mathcal{X}} \, \bar{F}_n(x|s),$$ and the true optimal decision $x^*(s)$ as $$ x^*(s) = \arg \min_{x \in \mathcal{X}} \, \E \left[ F(x,s,Z) \g S = s\right].$$  Let $\E\left[F(x,s,Z) \right] = F(x|s)$.  We would like the approximation to have the following property,
\begin{equation}\notag
\lim_{n\rightarrow\infty} x^*_n(s) = \lim_{n\rightarrow\infty} \arg \min_{x\in \mathcal{X}} \bar{F}_n(x|s) = \arg \min_{x\in\mathcal{X}} F(x|s) = x^*(s), \ \mathrm{almost \ surely.}
\end{equation}We discuss convergence properties in the following subsection.

\subsection{Convergence analysis}

Let $x_n^*(s)$ be the solution to Equation (\ref{eq:spoState2}) and let $x^*(s)$ be the true solution.  We would like $x_n^*(s)\rightarrow x^*(s)$ almost surely, pointwise in $s$, for $s$ in a compact subset of the state space, denoted $\mathcal{S}_D$.  In Theorem \ref{thm:samplePathConvergence}, we give conditions under which almost sure convergence occurs.

Before we state the main theorem, we give a set of assumptions.  

{\bf (A\ref{sec:samplePath}.1)}  $\mathcal{X}$ is a convex subset of $\R^d$ and $\mathcal{S}_D$ is a compact subset of $\mathcal{S}$.

{\bf (A\ref{sec:samplePath}.2)}  The function $F(x,s,Z(\omega))$ is almost surely convex and continuous in $x \in \mathcal{X}$ for every $s \in \mathcal{S}_D$.

{\bf (A\ref{sec:samplePath}.3)} For every fixed $s \in \mathcal{S}_D$ and $x \in \mathcal{X}$, let $(w_n(s,S_i)_{0:n})_{n=0}^{\infty}$ and the distribution of $(Z,S)$ be such that $$\lim_{n \rightarrow \infty} \sum_{i=0}^{n-1} w_n(s,S_i) F(x,S_i, Z(\omega_{i+1})) = F(x|s).$$

{\bf (A\ref{sec:samplePath}.4)} For every fixed $s \in \mathcal{S}_D$, the function $\E [F(x,s,Z(\omega)]$ has a unique minimizer.

Assumptions (A\ref{sec:samplePath}.1)--(A\ref{sec:samplePath}.2) places bounds on the decision set and state set. Assumption (A\ref{sec:samplePath}.3) places a pointwise convergence condition on the functional estimator; this places restrictions on both the weighting functions $w_n(s,S_i)$ and the distribution of $Z|S=s$.  See Section \ref{sec:DP} for a discussion on this assumption.  Assumption (A\ref{sec:samplePath}.4) assures that there is only one optimal decision per state.  The main convergence theorem for function-based optimization is as follows.

\begin{theorem}\label{thm:samplePathConvergence}
Let $\bar{F}_n(x|s) = \sum_{i=1}^n w_n(s,S_i) F( x,S_i,Z(\omega_{i+1}))$.  Suppose that assumptions (A\ref{sec:samplePath}.1)--(A\ref{sec:samplePath}.4) hold.  Then, $x_n^*(s) \rightarrow x^*(s)$ almost surely, pointwise for every $s \in \mathcal{S}_D$.
\end{theorem}

The proof of Theorem \ref{thm:samplePathConvergence} relies heavily on Proposition 2.4 and Theorem 3.7 of \citep{Ro96}.  We state them now, modified for our setting.

\begin{proposition}[Proposition 2.4 of \citep{Ro96}]
If, for every $s\in \mathcal{S}_D$, $\bar{F}_n(x|s)$ converges uniformly to $F(x|s)$ on all compact, non-empty subsets of $\mathcal{X}$, then $\bar{F}_n(x|s)$ epiconverges to $F(x|s)$.
\end{proposition}

\begin{theorem}[Corollary 3.11 of \citep{Ro96}]
Suppose that $\bar{F}_n(x|s)$ epiconverges to $F(x|s)$ and that $F(x|s)$ has a unique minimizer for a fixed $s \in \mathcal{S}_D$.  Then $x^*_n(s)$ converges almost surely to $x^*(s)$.
\end{theorem}

{\em Proof of Theorem \ref{thm:samplePathConvergence}.}  Fix $s$ in $\mathcal{S}_D$.  We show almost sure convergence by satisfying the conditions of Corollary 3.11 in \citep{Ro96}.  First, we show that $\bar{F}_n(x|s) \rightarrow F(x|s)$ uniformly for $x \in \mathcal{X}_D$, where $\mathcal{X}_D$ is a compact subset of $\mathcal{X}$.  Because $\bar{F}_n(x|s)$ is bounded and continuous, it is equicontinuous; because it is equicontinuous and converges pointwise in $x$ to $F(x|s)$, $\bar{F}_n(x|s)$ converges uniformly to $F(x|s)$.  Continuity of $F_n(x|s)$ and uniform convergence to $F(x|s)$ satisfy the conditions of Proposition 2.4 of \citep{Ro96}, which in turn satisfies the conditions of Corollary 3.11.  \qquad\endproof

We discuss the choice of weight functions in Section \ref{sec:DP}.  Before that, however, we present an algorithm and theoretical results for gradient-based optimization with a state variable.

\section{Gradient-based optimization with an observable state variable}\label{sec:sa}
In gradient-based optimization, we no longer observe an entire function $F(x,S_n,Z(\omega_{n+1}))$, but only a derivative taken at $x_i$, $$\hat{\beta}(x_n,S_n,\omega_{n+1}) =  \nabla_x F(x_n,S_n,Z(\omega_{n+1})).$$  Stochastic approximation is the most popular way to solve stochastic search problems using a gradient; it modifies gradient descent algorithms to account for random gradients~\citep{RoMo51,KiWo52}.  The general idea is to optimize $x$ by iterating,
\begin{equation}\label{eq:sa1}
x_{n+1} = \Gamma_{\mathcal{X}}\left(x_n - a_n \nabla_x \, F(x_n,Z(\omega_{n+1}))\right),
\end{equation}where $\Gamma_{\mathcal{X}}$ is a projection back into the constraint set $\mathcal{X}$, $\nabla_x \, F(x_n,Z(\omega_{n+1}))$ is the stochastic gradient at $x_n$ and $a_n$ is a stepsize.  Another approach to gradient-based optimization uses construction of piecewise linear, convex functions to approximate $F(x)$~\citep{GoPo01,PoRuTo04}; we will follow the second approach.

Including a state variable into gradient-based optimization is less straightforward than it is for function-based optimization.  We encounter difficulties because we choose $x_n$ given $S_n$; Equation (\ref{eq:sa1}) works because \textit{only} $x_n$ changes.  When we include an observed state $S_n$, the decision $x_n$ is based on the state $S_n$.  Therefore, it cannot be chosen in an iterative manner directly from $x_{n-1}$, which is based on the state $S_{n-1}$.  Additionally, constructing the approximate function $\bar{F}_n(x|s)$ in a convex manner is not trivial because the gradient observations are based on \textit{both} $x_n$ and $S_n$.  In this section, we give an algorithm for gradient-based optimization with a state variable, along with convergence analysis for that algorithm.

\subsection{Algorithm for gradient-based optimization}
We propose modeling $F(x|s)$ with a piecewise linear, convex, separable approximation.  Even if $F(x|s)$ is not itself separable, we aim to approximate it with a simpler (separable) function that has the same minimum for every fixed $s$.  Approximating the minimum well is easier than approximating the entire function~\citep{ChPo00,PoRuTo04}.  Moreover, convex interpolation is easier in one dimension than multiple dimensions. 
We approximate $\E[F(x,s,Z)]$ by a series of separable functions,
\begin{equation}\notag
\bar{F}_n(x|s) = \sum_{k=1}^d f_n^{k}(x^{k}|s),
\end{equation}where $x^{k}$ is the $k^{th}$ component of $x$ and $f^k_n(x|s)$ is a univariate, piecewise linear function in $x$.  We enforce convexity restrictions on marginal functions $f_n^{k}(x|s)$ for every $s \in \mathcal{S}$.  We assume the existence of stochastic gradients, $\hat{\beta}(x,s,\omega) = \nabla_x F(x,s,Z(\omega))$, which are obtained as a response instead of $F(x,s,Z(\omega))$.

The observations $(x_i,S_i,\hat{\beta}(x_i,S_i,\omega_{i+1}))_{i=0}^{n-1}$ are used to update $\bar{F}_n(x|s)$ sequentially.  We want to assemble a set of $d$ piecewise linear marginal functions $f^k_n(x|s)$ by constructing a series of slopes, $v_{0:n-1}^k(S_n)$, based on $\hat{\beta}_{1:n}$.  We use weights to group the gradients from states ``similar'' to $S_n$.  We outline the algorithm as follows.

\paragraph{Step 1: Observe $S_n$ and generate $(w_{n}(S_n,S_i))_{i=0}^{n-1}$}This is discussed in Section \ref{sec:DP}.

\paragraph{Step 2: Construct slopes for $f^k_n(x|S_n)$ given $\hat{\beta}_{1:n}$, $x_{0:n-1}$ and $(w_{n}(S_n,S_i))_{i=0}^{n-1}$}Fix $k$.  We begin by placing the observed decisions in ascending order: $$x_{[0]}^k\leq x_{[1]}^k\leq \dots \leq x_{[n-1]}^k,$$  where $[0],\dots,[n-1]$ is the ordered numbering.  A necessary and sufficient condition for $f^k_n(x|S_n)$ to be convex is for the slopes to be nondecreasing; that is, $$\frac{d}{dx} f^k_n(x|S_n) \leq \frac{d}{dy} f^k_n(y|S_n)$$ for every $x \leq y$.  We find a set of slopes $v^k_{n,[0]}(S_n) \leq \dots \leq v^k_{n,[n-1]}(S_n)$ corresponding to the ordered decisions $x_{[0]}^k,\dots,x_{[n-1]}^k$ using weighted least squares minimization, which is a quadratic program,
\begin{align}\label{eq:projection}
v^k_n(S_n)&= \arg \min_{v} \, \sum_{i=0}^{n-1} w_n\left(S_n,S_{[i]}\right)\left(\hat{\beta}(x_{[i]}^k,S_{[i]},\omega_{[i+1]}) - v_{[i]}\right)^2,\\\notag
\mathrm{subject \ to:} \ v_{[i-1]} & \leq v_{[i]}, \ \ \ i = 1,\dots, n-1.
\end{align}

\paragraph{Step 3: Reconstruct $f^k_n(x|S_n)$, $\bar{F}_n(x|s)$ given $v^k_n(S_n) $} Suppose that $\mathcal{X}$ is compact; there exists a minimum value $x^k_{min}$ and a maximum value $x^k_{max}$ for each dimension $k$.  Set $x_{[-1]}^k = x^k_{min}$ and $x_{[n]}^k = x^k_{max}$.  Define $f^k_n(x|S_n)$ as follows,
\begin{equation}\label{eq:reconstruct}
f^k_n(x|S_n) = \sum_{i = 0}^{\ell} v^k_{n,[i]}(S_n) \, (x_{[i]}^k - x_{[i-1]}^k ) + v^k_{n,[\ell]}(S_n) \, (x - x_{[\ell]}^k),
\end{equation}where $\ell$ is the smallest index such that $x_{[\ell]}^k \leq x < x_{[\ell+1]}^k$.    Set $$\bar{F}_n(x|S_n) = \sum_{k=1}^d f^k_n(x|S_n).$$
Note that the reconstruction is the same as the original up to a constant, which does not affect the optimal decision.
\paragraph{Step 4: Choose $x_n$ given $\bar{F}_n(x|S_n)$}
We want to choose an $x_n$ so that we learn as much as possible for an arbitrary $s$.  This is done by picking $x_n$ as follows,
\begin{equation}\label{eq:xnChoice}
x_n = \arg \min_{x \in \mathcal{X}} \, \bar{F}_n(x|S_n).
\end{equation}Note that $\bar{F}_n$ is a piecewise linear function; if $\mathcal{X}$ is a linear constraint set, the minimum can be found with a linear program.

\begin{algorithm}[h]\caption{Gradient-based optimization with an observable state variable}
  \begin{algorithmic}[1]\label{alg:online}
  \REQUIRE Query state $s$, initial slopes $v_0$.
  \FOR{$i=0$ to $n-1$}
  \STATE Observe random state $S_i$.
  \STATE Generate weights $\left(w_i(S_i,S_j)\right)_{j=0}^{i-1}$. (See Section \ref{sec:DP}.)
  \FOR{$k=1$ to $d$}
  	\STATE Place decision observations in ascending order:  $x_{[0]}^k\leq \dots \leq x_{[i-1]}^k.$  	\STATE Compute slopes $v_i^k(S_i)$ by {\normalsize \begin{align}\notag
& v^k_i(S_i)= \arg \min_{v}  \sum_{j=0}^{i-1} w_i\left(S_i,S_{[j]}\right)\left(\hat{\beta}(x_{[j]}^k,S_{[j]},\omega_{[j+1]}) - v_{[j]}\right)^2,\\\notag
& \mathrm{subject \ to:} \ v_{[j-1]} \leq v_{[j]}, \ \ \ j = 1,\dots, i-1.
\end{align}}
  	\STATE Reconstruct marginal function $f_i^{k}(x^k|S_i)$ using slopes $v_i^k(S_i)$ as per Equation (\ref{eq:reconstruct}). 
  \ENDFOR
  \STATE Set \begin{equation}\notag x_i = \arg\min_{x \in \mathcal{X}} \, \sum_{k=1}^d f_n^{k}(x^k|S_i).\end{equation}
  \STATE Observe random gradient $\hat{\beta}(x_i,S_i,\omega_{i+1}) = \nabla_x F(x_i, S_i, Z(\omega_{i+1})).$  
  \ENDFOR
  \STATE Compute $v_n^k(s)$, $k = 1,\dots, d$ as in Step 6.
  \STATE Compute $f_n^k(x^k|s)$, $k = 1,\dots,d$ using $v_n^k(s)$ as in Step 7.
  \STATE Set \begin{equation}\notag x_n^*(s) = \arg \min_{x \in \mathcal{X}} \, \sum_{k=1}^d f_n^{k}(x^k|s).\end{equation}
  \end{algorithmic}
\end{algorithm}

A full overview of this procedure is given in Algorithm \ref{alg:online}.  Notice that while the function-based optimization of Algorithm \ref{alg:function} can essentially be performed in a \textit{post hoc} batch setting, Algorithm \ref{alg:online} can only be performed in an online setting.  We discuss a grid-based extension of Algorithm \ref{alg:online} in the following subsection.

\subsection{Grid-based decisions}One of the more computationally heavy parts of Algorithm \ref{alg:online} is Step 6, the projection of the slopes to an ordered space via a quadratic program.  The number of parameters and constraints grows linearly with the number of observations.  In some numerical work, we have found it easier to make decisions on a grid format.

Fix $k$.  If $\mathcal{X}$ is compact, we can create an arbitrarily fine grid on the $k^{th}$ dimension with a finite number of points, $$a_{1}^k\leq\dots\leq a_{N(k)}^k.$$Suppose that they are evenly spaced with distance $\alpha$ and let the intersection of this set of points with $\mathcal{X}$ be denoted $\mathcal{X}^G$.  If all decisions are selected from $\mathcal{X}^G$, the parameter and constraint set for Equation (\ref{eq:projection}) never grows.

The inclusion of a grid changes the way that we select $x_n$.  Let $\hat{x}_n$ be the solution to the original approximated problem, $$\hat{x}_n = \min_{x\in\mathcal{X}}\bar{F}_n(x|S_n).$$We can generate $x_n$ from $\hat{x}_n$ in one of two ways: 1) projection to the nearest feasible point in $\mathcal{X}^G$, or 2) random selection of a neighboring point in $\mathcal{X}^G$.  The second option is computationally simple and can be guaranteed to break the constraints by at most an arbitrarily small amount through grid construction.  We use the second method in our numerical examples.

\subsection{Convergence analysis}
We now give conditions under which Algorithm \ref{alg:online} converges in probability to the global optimum pointwise for every state $s$ is a set of query states $\mathcal{S}_D$.  The observed state $S_n$ is likely not the same as our query state $s$, but we often care about what the approximation says is the best decision for $s$ after $n$ observations, defined as $x^*_n(s)$.  Set
\begin{equation}\notag
x^*_n(s) = \arg \min_{x \in \mathcal{X}} \, \bar{F}_n(x|s).
\end{equation}We would like $x_n^*(s)$ to approach the true optimal decision, $x^*(s)$, as $n \rightarrow \infty$, where
\begin{equation}\notag
x^*(s) = \arg \min_{x \in \mathcal{X}} \, \E \left[ F(x,s,Z)\g S = s\right] = \arg\min_{x\in\mathcal{X}} F(x|s).
\end{equation}  The outline of the proof is to show that the decisions sampled for a sufficiently small neighborhood of states around $s$ accumulates in an arbitrarily small neighborhood around $x^*(s)$.  This is done by verifying optimality conditions in the accumulation regions.  
First, however, we need to define a set of assumptions.


Suppose that there are a finite set of functions $g_i(x)$, $i = 1, \dots, p$ and $h_j(x)$, $j = 1,\dots, q$ such that the constraint set $\mathcal{X}$ can be written as $$\mathcal{X} = \left\{x \, : \, g_i(x) \leq 0, \ h_j(x) = 0, \ i = 1,\dots, d, \ j = 1,\dots, q\right\}.$$

The following conditions require separability and strong convexity of the objective function, along with differentiability of the objective function and constraints.

{\bf (A\ref{sec:sa}.1)} For every $s \in \mathcal{S}$, the function $F(x|s)$ is separable in $x$, $$F(x|s) = \sum_{k=1}^d f^k(x^k|s).$$  Define the gradient function $$v^k(x^k,s) = \frac{\partial}{\partial x^k} f^k(x^k|s).$$

{\bf (A\ref{sec:sa}.2)}  Let $F(x|s)$ be strongly convex in $x$ for every $s \in \mathcal{S}$ with parameter $m$; that is, $$\left( \nabla_x F(x|s) - \nabla_y F(y|s) \right)^T \, (x-y) \geq m || x - y ||^2_2.$$

{\bf (A\ref{sec:sa}.3)}  $F(x|s)$ is twice continuously differentiable in $x$ and $s$ for every $x \in \mathcal{X}$ and $s\in\mathcal{S}$.

{\bf (A\ref{sec:sa}.4)}  If the state space and decision space are well sampled around the point $(x,s)$ and every observation is unbiased, that is, $$\E \left[ \hat{\beta}(x',s') \right] = v(x',s'),$$ with $v(x,s)$ as in (A\ref{sec:sa}.1), then $$\lim_{n\rightarrow \infty} v_n(x,s) = v(x,s),$$ where $v_n(x,s)$ is defined by the projection in Equation (\ref{eq:projection}).

Now we discuss the optimality conditions.  Define $N_{\mathcal{X}}(x)$, the normal cone to $\mathcal{X}$ at $x$, $$N_{\mathcal{X}}(x) = \left\{y\in \R^d \g \langle y, x - v\rangle \geq 0, \, \forall v \in \mathcal{X}\right\}.$$Define the subgradient of $F(x|s)$ at $x$ as $\partial F(x|s)$.  Then, the point $x^*$ is a global minimizer of $F(x|s)$ for a fixed $s$ if and only if
\begin{equation}\label{eq:optimal}
0 \in \partial F(x^* | s) + N_{\mathcal{X}}(x^*).
\end{equation}We aim to show that as $n \rightarrow \infty$, $x^*_n(s)$ and only $x^*_n(s)$ satisfies Equation (\ref{eq:optimal}).

\begin{theorem}\label{thm:online}
Let $\mathcal{S}_D$ be a compact subset of $\mathcal{S}$ and assumptions (A\ref{sec:sa}.1)--(A\ref{sec:sa}.4) hold.  Then, for every $s \in \mathcal{S}_D$ and every $\epsilon > 0$, $$\p\left\{\left|x^*_n(s) - x^*(s) \right| > \epsilon \right\} \rightarrow 0$$ as $n \rightarrow \infty$.
\end{theorem}

\begin{proof}
Fix $s \in \mathcal{S}_D$ and $\epsilon > 0$.  
Consider the $k^{th}$ component of the decision variable.  Since $\mathcal{X}$ is compact, there exists an $x^k_{min} $ and $x^k_{max}$ such that $x^k_{min} \leq x^k \leq x^k_{max}$ for all $x \in \mathcal{X}$.  We place an $\epsilon/ 4md$-net on this axis such that $x^{*,k}(s)$ is in the center of one of the partitions, where $m$ is the strong convexity constant.  Label the regions $a_1^k, \dots, a_{M(k)}^k$.  This is done for all components of the decision variable, $k = 1, \dots, d$.

Note that for every iteration, the decision $x_n$ and the calculated optimum $x^*_n(s)$ are random variables.  Let $p_n^k(a_i,s)$ be the probability that $x^{*,k}_n(s) \in a_i^k$, \begin{equation}\notag
p_n^k(a_i,s) = \p \left\{ x^{*,k}_n(s) \in a_i^k\right\}, \ \ \ i = 1,\dots M(k).
\end{equation}Fix $\gamma > 0$.  Let $\mathcal{L}^k(s)$ be the set of partitions $a_i^k$ where there exists a subsequence $(n_j)_{j=1}^{\infty}$ such that
\begin{equation}\notag
\mathcal{L}^k(s) = \left\{ a_i^k \g \lim \inf_{n_j \rightarrow \infty} p_{n_j}^k(a_i^k,s) > \gamma\right\}.
\end{equation}That is, $\mathcal{L}^k$ is the set of all decisions sampled infinitely often with at least limiting probability $\gamma$.  Let $$A_{\epsilon}^k(s) = \left\{a_i^k \g  \left| a^k_i - x^{*,k}(s) \right| < \epsilon/d \right\}.$$  We will show that $\mathcal{L}^k(s)$ is non-empty and that $\mathcal{L}^k(s) \subset A_{\epsilon}^k(s)$; this is done in two phases.

{\bf Claim 1:  $\mathcal{L}^k(s) \neq \emptyset$.}  If $\gamma \leq 1/M(k)$, then at least one decision will be sampled infinitely often with at least probability $\gamma$.

{\bf Claim 2:  $\mathcal{L}^k(s) \subset A_{\epsilon}^k(s)$.}  Suppose $b \in \mathcal{L}^k(s)$ and $b \notin A^k_{\epsilon}(s)$.  Then, there exists an infinite subsequence $(n_i)_{i=1}^{\infty}$ such that $x_{n_i}^*(s) \in b$.  Since the marginal values, $v_n(x,s)$ are also continuous in $s$, there exists a set of $s'$ such that $\lim\inf p_{n_j}^k(b,s') > \gamma$ and $|v^k(x,s) - v^k(x,s') | < \epsilon/4md$ for every $x \in b$.  Denote this set by $B(s,b)$.  The states in set $B(s,b)$ will be sampled infinitely often, so by (A\ref{sec:sa}.2) and (A\ref{sec:sa}.4), there exists an $N$ such that for every $n \geq N$,
$$\left|v_n^k(x,s) - v(x,s) \right| < m*\epsilon/2md = \epsilon/2d.$$
But by (A\ref{sec:sa}.2) and (A\ref{sec:sa}.3), the function $x^*(s)$ is uniformly continuous in $s$ over $\mathcal{S}_D$.  Combining this fact with (A\ref{sec:sa}.1), for sufficiently small $\epsilon$,
$$0 \notin v_n(x,s) + N_{\mathcal{X}}(x)$$ for all $n \geq N$ and all $x \in b$.  Therefore, $b \notin \mathcal{L}^k(s)$, so Claim 2 is true and therefore the theorem holds.
\end{proof}

We now discuss the choice of weight functions.

\section{Weight functions}\label{sec:DP}
The choice of weight functions determines whether assumptions (A\ref{sec:samplePath}.3) and (A\ref{sec:sa}.4) are satisfied and which distributions of $F(x,s,\omega)$ satisfy them.  More importantly, however, the weight functions also determine how well $F_n(x|s)$ approximates $F(x|s)$ with finite sample sizes.  Before we discuss the specifics of individual weighting functions, let us discuss how weighting functions are constructed.

Weighting functions rely on density estimation procedures to approximate the conditional density $f(y|s)$, where $s$ is the state and $y$ is the response.  Then the mean conditional response $\E[Y | s] = \int y f(y|s)dy$ is calculated, which is the object of interest.  The conditional density is approximated by a weighted sum of observations, $\hat{f}_n(y|s)$; it is either expressly composed by observational weights, as in the case of kernel regression, or can be decomposed into observational weights, as in the case of Dirichlet process regression.

In this section, we discuss two weighting schemes in detail:  kernels and the Dirichlet process similarity measure.  Kernels been well studied and are easy to implement as a weighting scheme.  However, they often do not perform well with more complicated problems, such as those with a moderate or large number of covariates.  Therefore, we propose a Dirichlet process similarity measure when a richer class of weighting functions is required.  We discuss kernel weights in Subsection \ref{subsec:kernel} and the Dirichlet process similarity measure along with relevant background material in Subsection \ref{subsec:DP}.

\subsection{Kernel weights}\label{subsec:kernel}
Kernel weights rely on kernel functions, $K(s)$, to be evaluated at each observation to approximate the conditional density.  A common choice for $K$ with continuous covariates is the Gaussian kernel, $$K_h(s) = (2\pi h)^{-1/2} \exp\{-s^2/2h\},$$ where the variance $h$ is called the bandwidth.  Kernel weights have the advantage of being simple and easy to implement.  The simplest and most universally applicable weighting scheme is based on the Nadaraya-Watson estimator~\citep{Na64,Wa64}.  If $K(s)$ is the kernel and $h_n$ is the bandwidth after $n$ observations, define\begin{equation}\notag
w_n(s,S_i) = \frac{K_{h_n}\left(s-S_i\right)}{ \sum_{j=0}^{n-1}K_{h_n}\left(s-S_j\right)}.
\end{equation}  In the case of function-based optimization, the function estimate $\bar{F}_n(x|s)$ is
\begin{equation}\notag
\bar{F}_n(x|s) =\frac{\sum_{i=0}^{n-1} K_{h_n}\left(s-S_i\right) F\left(x,S_i,Z(\omega_{i+1})\right)}{\sum_{j=0}^{n-1}K_{h_n}\left(s-S_j\right)}.
\end{equation}
Kernel methods have few requirements to satisfy assumptions (A\ref{sec:samplePath}.3) and (A\ref{sec:sa}.4).  For function-based optimization, assumption (A\ref{sec:samplePath}.3) is satisfied if:
\begin{enumerate}
	\item $F(x,s,Z)$ has finite variance for every $x \in \mathcal{X}$, $s\in \mathcal{S}$.
	\item $F(x|s) = \E[F(x,s,Z)]$ is continuous in $s$ for every $x \in \mathcal{X}$.
\end{enumerate}Sufficient conditions for gradient-based assumption (A\ref{sec:sa}.4) are similar:
\begin{enumerate}
	\item $\hat{\beta}(x,s,Z)$ has finite variance for every $x \in \mathcal{X}$, $s\in \mathcal{S}$.
	\item $v(x,s) = \nabla_x\E[F(x,s,Z)]$ is continuous in $s$ and $x$.
\end{enumerate}The first condition is assured by assumption (A\ref{sec:sa}.3), so only the second condition must be checked.

Despite ease of use and a guarantee of convergence, kernel estimators require a well-sampled space, are poor in higher dimensions and are highly sensitive to bandwidth size.  There is a large literature on bandwidth selection~\citep{FaGi95,JoMaSh96}, but it is usually chosen by cross-validation.  To overcome many of these difficulties, we propose using a Dirichlet process mixture model over the states as an alternative weighting scheme.

\subsection{Dirichlet process weights}\label{subsec:DP}
One of the curses of dimensionality is sparseness of data.  As the number of dimensions grows, the distance between observations grows exponentially.  In kernel regression, this means that only a handful of observations have weights that are effectively non-zero.  Instead of basing weights on a distance that grows quickly with the number of dimensions, we would like to average responses for ``similar'' observations.  Regions of ``similarity'' can be defined with a clustering algorithm.  Dirichlet process mixture models (DPMMs) are Bayesian nonparametric models that produce a distribution over data partitions~\citep{Pi96,IsJa03}.  They were introduced in the 1970's~\citep{Fe73,BlMa73,An74}, but have gained popularity in the last fifteen years as more powerful computers have allowed posterior computation~\citep{Ma94,EsWe95}.  DPMMs have been used for classification~\citep{ShNe07}, clustering~\citep{MeSi02,DaMa05} and density estimation~\citep{EsWe95,MaMu98}, but we will use the partitioning feature directly.

Dirichlet process mixture models are easiest to understand in the context of density estimaiton.  The idea is that complicated distributions, such as the distribution of the state variable, can be modeled as a countable mixture of simpler distributions, such as Gaussians.  To make the model more flexible, the number of components is allowed to be countably infinite.  The Dirichlet process is a Bayesian model that places a prior on the mixing proportions, leading to a small number of components with non-trivial proportions.  Because an infinite number of components are allowed, the Dirichlet process circumvents the issue of determining the ``correct'' number of components, as is necessary in algorithms like k-means.

The clustering/partitioning property of the Dirichlet process is derived from the mixture assumption.  Two observations are in the same cluster or partition if they are generated by the same mixture component.  The query state $s$ is also placed in a cluster; the estimated response for $s$ is simply the average of all the responses associated with states in that cluster.  However, because the data labels, component locations and mixing proportions are not known, the Dirichlet process produces a distribution over clusterings.  We use Monte Carlo methods to integrate over the clusterings.

In this subsection, we discuss the basic properties of Dirichlet process mixture models, how they can be used to generate a weighting function, how it can be approximated, and finally what is required for it to satisfy the optimization algorithm assumptions.






\subsubsection{Dirichlet process mixture models}
A mixture model represents a distribution, $g_0(s)$, as a weighted sum of simpler distributions, $g(s\g \theta_i)$, which are parameterized by $\theta_i$,
\begin{equation}\notag
g_0(s) = \sum_{i=1}^{K} p_i g(s \g \theta_i).
\end{equation}Here, $p_i$ is the mixing proportion for component $i$.  For example, if $g_0(s)$ is a univariate, continuous distribution, we may wish to represent it as a sum of Gaussian densities,
\begin{equation}\label{eq:gsnMixture}
g_0(s) = \sum_{i=1}^{K} p_i \, \frac{1}{\sqrt{2\pi \sigma_i^2}} e^{-\frac{1}{2 \sigma_i^2} (s-\mu_i)^2}.
\end{equation}In Equation (\ref{eq:gsnMixture}), $g$ is the Gaussian density; it is parameterized by $\theta_i = (\mu_i, \sigma_i^2)$, the mean and variance for component $i$. 

The difficulties of using a mixture model are 1) determining parameters $(p_i,\theta_i)_{i=1}^K$, and 2) determining $K$.  There are many optimization based algorithms to find $(p_i,\theta_i)$ (see \citep{HaTiFr09} for a review), but fewer ways to find a good value for $K$.  However, if we assume $K = \infty$ with only a finite number of components with weights $p_i$ that are effectively non-zero, then we effectively do not have to choose $K$.

We can use a Dirichlet process (DP) with base measure $\G_0$ and concentration parameter $\alpha$ to place a distribution over the joint distribution of $(p_i, \theta_i)$, the mixture proportion and location of component $i$~\citep{Fe73,An74}.  A Dirichlet process effectively places a discrete probability distribution over the parameter ($\theta$) space; the $\theta$'s that are given positive probability in that space are components of the mixture model.  The probability associated with $\theta$ is the component weight $p$.  We shall call this distribution over the parameter space $P$.  Note that a Dirichlet process prior does not tell us deterministically what $(p_i,\theta_i)$ \textit{will} be; instead it places a distribution over what it \textit{could} be.  For that reason, $P$ is actually a random measure.  A Dirichlet process mixture model is constructed as follows.

Assume that data $S_1,\dots,S_n$ are drawn from the same distribution, which is modeled by a mixture over distribution $G(\theta)$.  We let $g(\cdot \g \theta)$ be the density, while $G(\theta)$ is the distribution, for example $N(\mu,\sigma^2)$.  Observation $S_i$ is drawn from a component of that model, $G(\theta_i)$, with parameter $\theta_i$.  Conditioned on $\theta_i$, $S_i$ has the distribution $G(\theta_i)$.  Now, let $P$ be the mixing distribution over $\theta$; we give $P$ a Dirichlet process prior with base distribution $\G_0$ and concentration parameter $\alpha$.  In sum, this produces a Dirichlet process mixture model.  We use the following hierarchical model for the DPMM,
\begin{align}\label{eq:mixture}
P & \sim DP(\alpha, \G_0),\\\notag
\theta_i | P & \sim P,\\\notag
S_i | \theta_i & \sim G(\theta_i).
\end{align}Here, ``X $\sim$ Y'' means ``$X$ has the distribution of $Y$.''  Note the conditional independence at every level of the Model (\ref{eq:mixture}); for example, given $\theta_i$, $S_i$ is independent of $P$ and the other $S_j$.  Distributions $F$ and $\G_0$ often depend on additional hyperparameters; these will be explained in context later. 

A Dirichlet process is used as a prior on $P$ because it produces an almost surely discrete distribution over parameters.  This is demonstrated when we integrate out $P$ from Model (\ref{eq:mixture}) to obtain a conditional distribution of $\theta_n | \theta_{1:n-1}$~\citep{BlMa73}
\begin{equation}\label{eq:polya}
\theta_n \g \theta_1, \dots, \theta_{n-1} \sim \frac{1}{\alpha + n -1} \sum_{i=1}^{n-1} \delta_{\theta_i} + \frac{\alpha}{\alpha + n - 1} \G_0.
\end{equation}Here, $\delta_{\theta}$ is the Dirac measure with mass at $\theta$.  Equation (\ref{eq:polya}) is known as a \textit{Polya urn posterior}.  Note that $\theta_n$ has positive probability of assuming the value of one of the previously observed $\theta_i$, but it also can take a completely new value drawn from $\G_0$ with positive probability.  The parameter $\alpha$ controls how likely $\theta_n$ is to take a new value.

We now give an example of a DPMM.  Suppose that $g_0(s)$ is univariate and continuous.  An infinite Gaussian mixture model is a good approximation, parameterized by $\theta_i = (\mu_i,\sigma_i^2).$  Let $S_1, \dots, S_n$ be drawn from this distribution.  The mixture model can be written as,
\begin{align}\label{eq:gsnModel}
P & \sim DP(\alpha, \G_0),\\\notag
\theta_i=(\mu_i,\sigma_i^2) \g P & \sim P,\\\notag
S_i \g \theta_i & \sim N(\mu_i, \sigma_i^2).
\end{align}Often, $\G_0$ is chosen to be conjugate to $G$ to ease posterior sampling; in this case, the conjugate $\G_0$ is Normal-Inverse-Gamma with hyperparameters $(\lambda_0,\nu_0, \alpha_0, \beta_0).$  $\alpha$ is also a hyperparameter; it is usually given a Gamma prior or set equal to 1.  We now discuss DPMM weights.

\subsubsection{The Dirichlet process weights}\label{sec:DPsimilarity}
Dirichlet process mixture models intrinsically produce a partition structure~\citep{Pi96,IsJa03}, that is, a clustering of the observed data.  We can see this in the Polya urn posterior of Equation (\ref{eq:polya}); each hidden parameter has positive probability of taking the same value as another parameter.  If two parameters have the same value, the associated observations are in the same partition/cluster.  Because a Dirichlet process places a distribution over component parameters and probabilities, which form a partition of the data, it also places a distribution over all partition structures of the data.  We use the partition structure to induce weights on the observations by giving equal weights to all observations that are in the same partition as the current observed state.

Let the cluster/partition $C_i$ be defined as the set of all observations that have the same parameter, $C_i = \{j : \theta_j = \theta_i^*\}$.  Let $\pp = \{ C_1, \dots, C_{n(\pp)}\}$ be the partition of the observations $\{1,\dots,n\}$.  Given a partition $\pp$, there are $n(\pp)$ clusters, generating $n(\pp)$ unique parameter values, $\theta_1^*,\dots,\theta_{n(\pp)}^*$.  Now suppose that we know the partition $\pp$.  Given this partition, we wish to place our observed state $s$ into one of the partitions.  We do not know its partition, but we can generate a probability that it is in cluster $C_i$,
\begin{align}\notag
p_s(C_i|\pp) &= \p(s \in C_i \g \pp, S_{1:n}) \\\label{eq:assign}
& \propto |C_i| \int g(s\g \theta^*) dH_{C_i} (\theta^*),
\end{align}where $|C_i|$ is the number of elements in $C_i$, $H_{C_i}(\theta^*)$ is the posterior distribution of $\theta^*$ conditioned on the base measure $\G_0$ and the set of observations $\{S_j  \, : \, S_j \in C_i \}$.  Sometimes it is impossible to compute the integral in Equation (\ref{eq:assign}) and it is approximated by Monte Carlo integration conditioned on $\theta^*$.

The probability is calculated for each cluster $C_i$, $i = 1,\dots, n(\pp)$.  Given the probabilities $(p_s(C_i|\pp))_{i=1}^{n(\pp)}$, the weighting function is defined by
\begin{equation}\label{eq:condDPW}
w_n(s,S_i) \g \pp = \sum_{j = 1}^{n(\pp)} \frac{p_s(C_j\g \pp)}{|C_j|} \1_{\{S_i \in C_j\}}.
\end{equation}Equation (\ref{eq:condDPW}) is different from the conditional Polya urn posterior of Equation (\ref{eq:polya}):  unlike the observed states, the query state is not allowed to be in a cluster by itself.  We do this because we often do not have prior information on the response distribution.

Equation (\ref{eq:condDPW}) is conditioned on a partition structure, but the Dirichlet process produces a distribution over partition structures.  Let $\pi(\pp)$ be the partition probability function, which is a prior distribution for partitions $\pp$, 
\begin{equation}\label{eq:eppf}
\pi(\pp) = \frac{\alpha^{n(\pp)-1} \prod_{j=1}^{n(\pp)} (|C_j|-1)!}{\prod_{j=1}^{n-1}(\alpha+j)}.
\end{equation}The posterior distribution, $\pi(\pp|S_{1:n})$ has the form
\begin{equation}\label{eq:partitionPosterior}
\pi(\pp|S_{1:n}) \propto \pi(\pp) \prod_{j=1}^{n(\pp)} \int_{\T} \prod_{i \in C_j} g(S_i | \theta) \, \G_0(d\theta).
\end{equation}We can combine Equations (\ref{eq:condDPW}) and (\ref{eq:partitionPosterior}) to obtain unconditional weights,
\begin{equation}\label{eq:dpWeights}
w_n(s,S_i) = \sum_{\pp} \pi(\pp|S_{1:n}) \left(\sum_{j = 1}^{n(\pp)} \frac{p_s(C_j\g \pp)}{|C_j|} \1_{\{S_i \in C_j\}}\right).
\end{equation}The conditional weights in Equation (\ref{eq:condDPW}) are easy to compute, but it is nearly impossible to compute $\pi(\pp|S_{1:n})$ or even enumerate all possible partitions, $\pp$.  Therefore, we approximate Equation (\ref{eq:dpWeights}) by performing a Monte Carlo integration over the partitions.  We obtain $M$ i.i.d. posterior partition samples, $(\pp^{(m)})_{m=1}^M$ and set
\begin{equation}\label{eq:dpWeightsApprox}
w_n(s,S_i) \approx \frac{1}{M} \sum_{m=1}^M \sum_{j = 1}^{n(\pp^{(m)})} \frac{p_s(C_j\g \pp^{(m)})}{|C_j|} \1_{\{S_i \in C_j\}}.
\end{equation}We now show how to obtain $(\pp^{(m)})_{m=1}^M$ given $S_1,\dots,S_n$ using Gibbs sampling.

\subsubsection{Gibbs Sampler for the State Variable}\label{sec:Gibbs}
Markov Chain Monte Carlo (MCMC)~\citep{Ne00} is the most popular and simple way to obtain partition structure samples, $(\pp^{(m)})_{m=1}^M$.  MCMC is based on constructing a Markov chain that has a limiting distribution equal to the partition structure posterior distribution.  The most common way to implement MCMC is by Gibbs sampling.  In Gibbs sampling, the partition $\pp=\{C_1,\dots,C_{n(\pp)}\}$ and possibly the parameters $\theta^* = (\theta_1^*,\dots,\theta_{n(\pp)}^*)$ form the state of the Markov chain.  If $\G_0$ is a conjugate prior for $G$, then $\theta^*$ is not needed and the sampler is called ``collapsed''; otherwise, $\theta^*$ is included.  Every iteration we choose an $i$ sequentially, with $1\leq i \leq n$, and remove $S_i$ from the clustering.  Then, we randomly assign it to either 1) one of the existing clusters, or 2) to a new cluster.  The assignment probabilities are chosen such that the limiting distribution of the Markov chain is the posterior distribution $\pi(\pp\g S_{1:n})$.  We follow Algorithm 3 of \citep{Ne00}; this algorithm is designed for problems with base measure $\G_0$ conjugate to the state conditional distribution $G$.  However, efficient algorithms for non-conjugate base measures can also be found in \citep{Ne00}. 

Let $c_i$ be the partition number for observed state $S_i$ and let $n_c$ be the number of observations in cluster/partition $c$.  The partition $\pp$ can readily be reconstructed from $c_1,\dots,c_n$; $c_1,\dots,c_n$ have the same relation to $C_1,\dots, C_{n(\pp)}$ that $\theta_1,\dots,\theta_n$ have to $\theta^*_1,\dots,\theta^*_{n(\pp)}$.  That is, $C_i = \{j \, : \, c_j = i\}.$

Gibbs sampling repeatedly samples a Markov chain where the limiting distribution is the posterior distribution of Equation (\ref{eq:mixture}).  The state of the chain is $\mathbf{c} = (c_1,\dots,c_n)$.  We move around the space by changing the partition for one observed state $S_i$ probabilistically while holding the partitions of all the other states fixed.  Let $$\mathbf{c}_{-i}=(c_1,\dots,c_{i-1},c_{i+1},\dots,c_n),$$ and $n_{-i,c}$ be the number of observations in partition $c$ when $i$ has been removed.  Fixing $\mathbf{c}_{-i}$, we compute the transition probabilities for $c_i$ as follows:
\begin{align}\notag
\p(c_i = c | \mathbf{c}_{-i}, S_i)& =  b \frac{n_{-i,c} }{n + \alpha} \int g(S_i| \theta^*) dH_{-i,c}(\theta^*),& \ \mathrm{ if } \ c=c_j \ \mathrm{ for } \ j \neq i,\\\notag
\p(c_i \neq c_j \ \forall \, j \neq i |   \mathbf{c}_{-i}, S_i)& = b \frac{\alpha}{n + \alpha} \int g(S_i|\theta^*) d\G_0(\theta^*), & \ \mathrm{ otherwise.}
\end{align}Here, $b$ is a normalizing constant, $H_{-i,c}(\theta^*)$ is the posterior distribution conditioned on the base measure $\G_0$ and set of observations $\{S_j : \theta_j = \theta_c^*, \, j \neq i \}$.  Because $\G_0$ is conjugate a conjugate prior for $g$, $H_{-i,c}(\theta^*)$ has a closed form solution~(see \citep{GeCaSt04} for a comprehensive list of posterior forms).  The chain is run until some convergence criteria are met.  Convergence is notoriously hard to diagnose, but general rules of thumb include setting a very large number of ``burn-in'' iterations and discarding all observations before the burn-in or running a number of chains and comparing the within and between sequence parameter variance or posterior probability.  See \citep{GeCaSt04} for a more thorough discussion of convergence criteria.  The Gibbs sampler for $(\pp^{(m)})_{m=1}^M$ is given in Algorithm \ref{alg:weights}.

\begin{algorithm}[t]
\caption{Gibbs sampler for $\pp|S_{1:n}$ with a conjugate base measure}
\label{alg:weights}
\begin{algorithmic}[1]
\REQUIRE Observed states $S_1,\dots,S_n$.
\STATE Initialize $\mathbf{c}$, set $m = 1$.
\WHILE{$m \leq M$}
\FOR{$i = 1$ to $n$}
  \STATE Sample $c_i | \mathbf{c}_{-i}$ using transition probabilities\begin{align}\notag
  \p(c_i = c | \mathbf{c}_{-i}, S_i)& =  b \frac{n_{-i,c} }{n + \alpha} \int g(S_i| \theta^*) dH_{-i,c}(\theta^*), \ c=c_j, \, j \neq i,\\\notag
\p(c_i \neq c_j \, \forall \, j|   \mathbf{c}_{-i}, S_i)& = b \frac{\alpha}{n + \alpha} \int g(S_i|\theta^*) d\G_0(\theta^*), \ \ \ \ \, \mathrm{ otherwise}
\end{align}
\ENDFOR
\IF{Convergence criteria are satisfied}
  \STATE Construct $\pp$ from $\mathbf{c}$, $$C_i = \{j \, : \, c_j = i\}, \ \ i = 1,\dots, \max(\mathbf{c}).$$
  \STATE Set $\pp^{(m)} = \pp$.
  \STATE Set $m = m+1$.
\ENDIF
\ENDWHILE
\end{algorithmic}
\end{algorithm}

We now discuss under which conditions Dirichlet process weights satisfy the convergence conditions for stochastic search with a state variable.

\subsubsection{Convergence conditions}Assumptions (A\ref{sec:samplePath}.3) and (A\ref{sec:sa}.4) concern the consistency of the regression estimator, which in turn depends on the underlying observation distribution and weights.  Dirichlet process weights also produce a density estimate; weak consistency of this density estimate is enough to satisfy the assumptions.

Posterior consistency is the notion that the posterior distribution of the DP mixture model in Equation (\ref{eq:mixture}) accumulates in neighborhoods ``close'' to the true distribution of the observations.  For weak consistency, we would like it to accumulate in weak neighborhoods.  

Weak consistency for DPMMs depends on both the model and the base measure; it has been examined in numerous articles~\citep{BaScWa99,GhGhRa99,GhRa03,Wa04,To06,RoDuGe09}.  Gaussian DPMMs with a conjugate base measure, here the Normal-Inverse Wishart, are weakly consistent for many continuous densities.  See \citep{GhGhRa99}, \citep{Wa04}, \citep{To06} and \citep{RoDuGe09} for conditions.  \citep{GhRa03} also show that DPMMs are consistent for finite distributions provided that the base measure $\G_0$ gives full support to the required probability simplex.

Assumption (A\ref{sec:samplePath}.3) is satisfied if:
\begin{enumerate}
	\item The DPMM and base measures are weakly consistent for the true state distribution $g_0(s)$.
	\item The distribution of $F(x,s,Z)$ is weakly convergent in $s$ for every $x\in \mathcal{X}$.  That is, $F(x,s',Z) \Rightarrow F(x,s,Z)$ as $s'\rightarrow s$, where ``$\Rightarrow$'' denotes weak convergence.
	\item  $F(x,s,Z)$ is almost surely bounded and continuous in $x$ and $s$.
\end{enumerate}These three conditions combine to produce a weakly convergent Bayes estimate of the conditional density, which is generated by the Dirichlet process weights on the observations.  These are much heavier conditions than those for kernel-based weights.  

Because random sampling is required in all current weak consistency results for DPMMs, it is an open question under which conditions Dirichlet process weights satisfy assumption (A\ref{sec:sa}.4).  We now study function-based optimization, gradient-based optimization and the weighting functions empirically.

\section{Empirical analysis}\label{sec:numbers}
We analyzed the performance of function-based and gradient-based optimization algorithms in conjunction with kernel- and Dirichlet-based weights on the hour ahead wind commitment problem and the two-product newsvendor problem.

\subsection{Multi-product constrained newsvendor problem}
A multi-product newsvendor problem is a classic operations research inventory management problem~\citep{PeDa99}.  In the two product problem, a newsvendor is selling products $A$ and $B$.  She must decide how much of each product to stock in the face of random demand, $D_A$ and $D_B$.  $A$ and $B$ can be be bought for $(c_A,c_B)$ and sold for $(p_A,p_B)$, respectively.  Any inventory not sold is lost.  Let $(x_A,x_B)$ be the stocking decisions for $A$ and $B$ respectively; it is subject to a budget constraint, $b_A \, x_A + b_B \, x_B \leq b$, and a storage constraint, $r_A \, x_A + r_B \, x_B \leq r.$  An observable state $S=(S_1,S_2)$ contains information about $D_A$ and $D_B$.  The problem is,
\begin{align}\label{eq:newsvendor}
 \max_{x_A, \, x_B}  & \, - c_A \, x_A - c_B \, x_B + \E\left[ p_A \min\left(x_A,D_A\right) + p_B \min \left(x_B, D_B\right)| S = s\right]\\\notag
\mathrm{subject \ to:} \ & b_A \, x_A + b_B \, x_B \leq b, \\\notag
&  r_A \, x_A + r_B \, x_B \leq r.
\end{align}

We generated data for Problem (\ref{eq:newsvendor}) in the following way.  
Demand and two state variables were generated in a jointly trimodal Gaussian mixture with parameters $D_A = 1/3*[N(10,4)+N(28,5)+N(30,5)]$, $D_B = 1/3*[N(10,3)+N(22,9)+N(35,12)]$; there were two state variables, $S_1$ and $S_2$; parameters were also included in the Gaussian trimodal mixture; all parameters were generated as follows: $\mu_{a,i} \sim N(0,3), \, \sigma_{a,i}^2) \sim Inverse \, Gamma(1,1)$, $a = A,B$, $i = 1,2,3$.  

\begin{figure}[t]\label{fig:newsvendorState}
\begin{center}
\includegraphics*[width=6.5in]{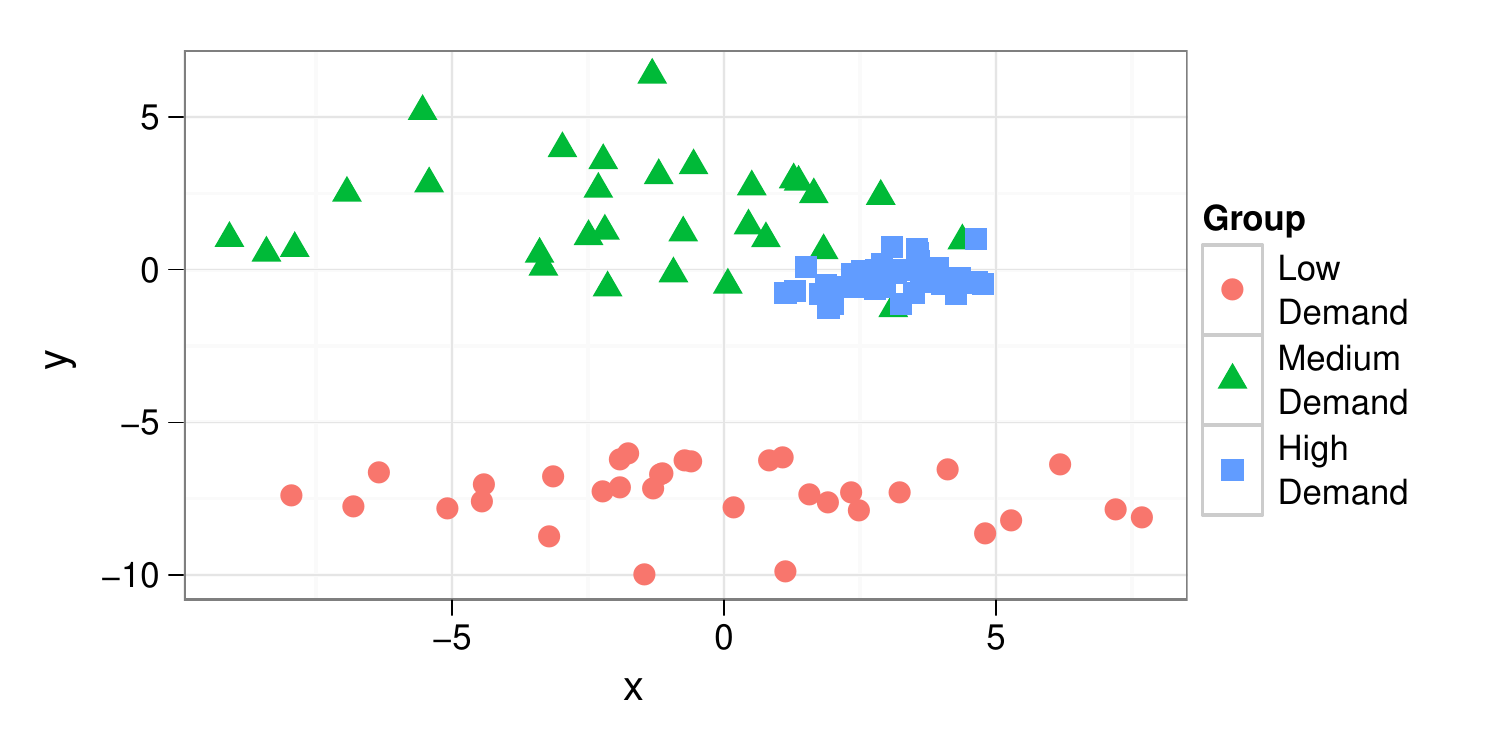}
\end{center}
\caption{State distribution labeled by demand level (low, medium, high).}
\end{figure}

\begin{figure}[t]\label{fig:newsvendor}
\begin{center}
\includegraphics*[width=6.5in]{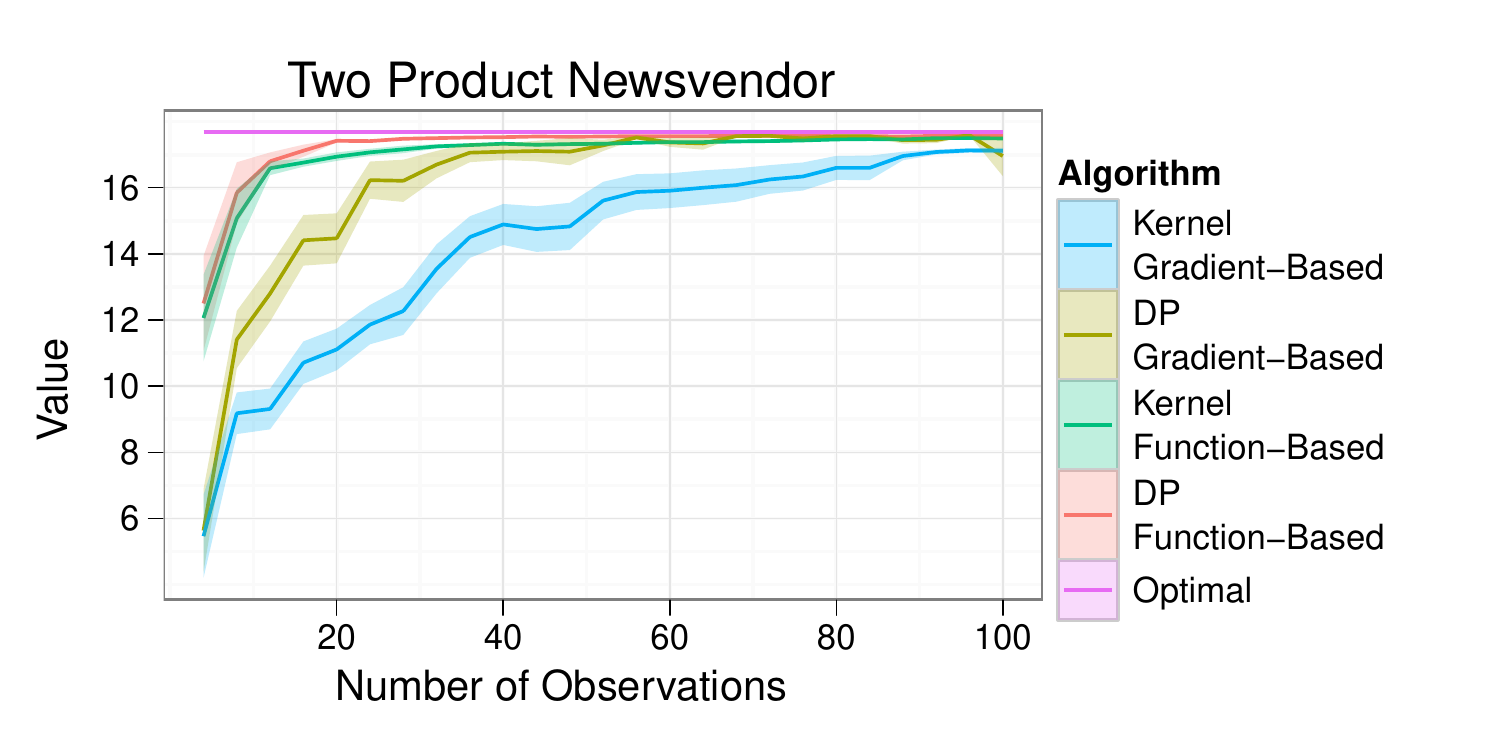}
\end{center}
\caption{Gradient-based  and function-based methods as a function of number of data points sampled.  Results are averaged over 100 test problems with observed demand.}
\end{figure}

\subsubsection{Newsvendor Competitors} The following methods were compared:
\begin{enumerate}
	\item {\it Function-based with kernel.}  Bandwidth is selected according to the ``rule of thumb'' method of the \textit{np} package for R, $$h_j = 1.06 \sigma_j n^{-1/(4 + d)},$$ where $\sigma_j$ is defined as min(sd, interquartile range/1.349)~\citep{HaRa08}.
	\item {\it Gradient-based with kernel.}  Bandwidth selection was the same as in the function-based case; decisions were made online after an initialization period of 5 random decisions.
	\item {\it Function-based with Dirichlet weights.}  We used the following hierarchical model,
\begin{align}\label{eq:newsvendorDP}
P & \sim DP(\alpha, \G_0), \\\notag
\theta_i = (\mu_{i}, \sigma_{i}^2)| P & \sim P, \\\notag
S_{i,j} | \theta_i & \sim N(\mu_{i,j},\sigma_{i,j}^2), \ j = 1,2.\end{align}Conjugate base measures were used.  Posterior samples were drawn using Gibbs sampling with a fully collapsed sampler run for 500 iterations with a 200 iteration burn-in with samples taken every 5 iterations.  
	\item {\it Gradient-based with Dirichlet weights.}  Dirichlet process mixture model was as in Model \ref{eq:newsvendorDP}; base measures and sampling procedures were the same.  Decisions were made online after an initialization period of 5 random decisions.
	\item {\it Optimal.}  These are the optimal decisions with known mixing parameters and unknown components.  That is, we know that the distribution of $D_A$ is $1/3*[N(10,4)+N(28,5)+N(30,5)]$, we have similar knowledge for $D_B$, $S_1$ and $S_2$ and we know their joint, but we do not know from which component (1, 2 or 3) the observation was drawn.  Decisions were made by solving the unconstrained newsvendor problem, projecting onto the constraint set (if necessary) and then performing a boundary search until the the optimal decision was reached.
\end{enumerate}

\subsubsection{Newsvendor Results}  Decisions were made under each regime over eight sample paths; 100 test state/demand pairs were fixed and decisions were made for these problems given the observed states/decisions in the sample path for each method.  The state distribution is shown in Figure \ref{fig:newsvendorState}.  Results are given in Figure \ref{fig:newsvendor}.  The kernel and Dirichlet process weights performed approximately equally for each method.  Function-based methods did better than gradient-based methods as the function-based methods used more available information.  Performance of Dirichlet weights versus kernel weights depends mainly on the underlying state distribution; there are cases for both in which one is preferable to the other.

\subsection{Hour ahead wind commitment}\label{sec:windnumbers}
The details of the hour ahead wind commitment problem from Section \ref{sec:intro} are as follows.  A wind farm manager must decide how much energy to promise a utility an hour in advance, incorporating knowledge about the current state of the world.  The decision is the amount of wind energy pledged, a scalar variable.  If more energy is pledged than is generated, the difference must be bought on the regulating market, which is expensive with a price that is unknown when the decision is made; otherwise, the excess is lost.  The goal is to maximize expected revenue.  The observable state variable is the time of day, time of year, wind history from the past two hours, contract price and current regulating price,
\begin{center}
\begin{tabular}{r @{}c@{} l r @{}c@{} l }
  $T^D_i$\ \ &$=$& \ time of day, & $T^Y_i$\ \ & $=$ & \ time of year,  \\
  $P^R_i$\ \ &$=$& \ current spot price, & $P^C_i$ \ \ & $=$ & \ contract price,\\
  $W_{i-1}$\ \ &$=$&\ wind speed an hour ago, & $W_{i}$ \ \ & $=$ & \ current wind speed,\\
  $S_i$\  \ &$=$& \ observable state variable & & $=$ & \ $ (T^D_i, T^Y_i, P^C_i, P^S_i,W_i,W_{i-1})$,\\
  $x_i$\ \ &$=$& \ amount of energy pledged, &  $Y_{i+1}(x)$\ \ & $=$ & \ $P^C_i \, x \, - P^S_{i+1} \, \max \left(x - W_{i+1}, 0\right)$.
\end{tabular}
\end{center}The revenue that the wind farm receives, $Y_{i+1}(x)$, depends on the variables $P^S_{i+1}$ and $W_{i+1}$, which are not known until the next hour.  We used wind speed data from the North American Land Data Assimilation System with hourly observations from 2002--2005 in the following locations:
\begin{enumerate}
	\item {\it Amarillo, TX.} Latitude: 35.125 N, Longitude: 101.50 W.  The data have strong daily and seasonal patterns.  The mean wind level is 186.29 $(m/s)^3$ with standard deviation 244.86.
	\item {\it Tehachapi, CA.} Latitude: 35.125 N, Longitude: 118.25 W.  The data have strong seasonal patterns.  The mean wind level is 89.45 $(m/s)^3$ with standard deviation 123.47.
\end{enumerate}

Clean regulating and contract price data for the time period were unavailable, so contract prices were generated by Gaussian random variables with a mean of 1 and variance of 0.10.  Regulating prices were generated by a mean-reverting (Ornstein-Uhlenbeck) process with a mean function that varies by time of day and time of year~\citep{Sc97}.  The data were analyzed separately for each location; they were divided by year, with one year used for training and the other three used for testing.  

\subsubsection{Wind Competitors} The following methods were compared on this dataset:

\begin{enumerate}
	\item {\it Known wind.}  The wind is known, allowing maximum possible commitment, $x_i = W_{i+1}(\omega_{i+1}).$  It serves as an upper bound for all of the methods.

	\item {\it Function-based with kernel weights.}  Function-based optimization where the weights are generated by a Gaussian kernel.  Bandwidth is selected according to the ``rule of thumb'' method of the \textit{np} package for R, $$h_j = 1.06 \sigma_j n^{-1/(4 + d)},$$ where $\sigma_j$ is defined as min(sd, interquartile range/1.349)~\citep{HaRa08}.

	\item {\it Function-based with Dirichlet process weights.}  Function-based optimization with Dirichlet process based weights.  We model the state distribution with the following hierarchical model,
\begin{align}\notag
P & \sim DP(\alpha, \G_0), &
\theta_i| P & \sim P,\\\notag
T_i^D |\theta_i & \sim \mathrm{von \ Mises}(\mu_{i,D},\phi_D),&
T_i^Y | \theta_i & \sim \mathrm{von \ Mises}(\mu_{i,Y},\phi_Y),\\\notag
P_i^C | \theta_i & \sim N(\mu_{i,C},\sigma^2_{i,C}),&
P_i^R | \theta_i & \sim N(\mu_{i,R},\sigma^2_{i,R}),\\\notag
W_i | \theta_i & \sim N(\mu_{i,W1},\sigma_{i,W1}^2),&
W_{i-1} | \theta_i & \sim N(\mu_{i,W2},\sigma_{i,W2}^2),
\end{align}\vspace{-0.35in}
\begin{equation}\notag
\theta_i = (\mu_{i,D}, \mu_{i,Y}, \mu_{i,C},\sigma_{i,C}^2, \mu_{i,S}, \sigma_{i,S}^2, \mu_{i,W1}, \sigma_{i,W1}^2, \mu_{i,W2},\sigma_{i,W2}^2).
\end{equation}We modeled the time of day, $T_i^D$, and year, $T_i^Y$, with a von Mises distribution, an exponential family distribution over the unit sphere; the dispersion parameters, $\phi_D$ and $\phi_Y$, are hyperparameters.  The base measure was Normal-Inverse Gamma for $P_i^C$, $P_i^R$, $W_i$ and $W_{i-1}$ and uniform for the means of $T_i^D$ and $T_i^Y$. 100 posterior samples were drawn using Gibbs sampling with a collapsed sampler for all conjugate dimensions after a 1,000 iteration burn-in and 10 iteration pulse between samples.  

	\item{\it Ignore state.}  Sample average approximation is used, $$\bar{F}_n(x|s) = \frac{1}{n} \sum_{i=0}^{n-1} Y_{i+1}(x).$$
\end{enumerate}
\begin{table}[t]
\begin{center}
\begin{small}
\begin{sc}
\begin{tabular}{l  @{}r @{}r @{}r @{}r @{}r @{}r @{}r @{}r }
  \hline
  Method/Location & \multicolumn{2}{c}{2002} &\multicolumn{2}{c}{2003} & \multicolumn{2}{c}{2004} & \multicolumn{2}{c}{2005} \\\hline
   \textbf{Tehachapi, CA} & & & & & & & & \\
  Known Wind & 97.5 & & 94.5& & 73.7& & 91.8& \\
  Function with Kernel & 78.8&\, \textit{(80.8\%)} & 77.3&\, \textit{(81.8\%)}& 58.9 &\, \textit{(79.9\%)}& 72.1 &\, \textit{(78.5\%)}\\
  Function with DP & {\bf 85.1}& {\bf \textit{(87.3\%)}} & {\bf 82.6} & {\bf\textit{(87.4\%)}}& {\bf63.9}  & {\bf \textit{(86.7\%)}} & {\bf79.6} & {\bf \textit{(86.7\%)}}\\
  Ignore State & 30.4 & \textit{(31.1\%)} & 31.1 & \textit{(32.9\%)}& 22.8&  \textit{(30.9\%)}& 29.3&  \textit{(31.9\%)}\\\hline
   \textbf{Amarillo, TX} & & & & & & & & \\
  Known Wind & 186.0& & 175.2& & 184.9& & 175.2& \\
  Function with Kernel & 155.1 & \textit{(83.4\%)} & 149.6&  \textit{(85.4\%)}& 154.7& \textit{(83.7\%)} &146.2& \textit{(83.5\%)} \\
  Function with DP & \ {\bf168.2} & {\bf \textit{(90.4\%)}} & \ {\bf160.6} & {\bf \textit{(91.7\%)}}&\ {\bf167.1} & {\bf \textit{(90.4\%)}} & \ {\bf159.4} & {\bf \textit{(91.0\%)}}\\
  Ignore State & 70.3 & \textit{(37.8\%)} & 68.7 & \textit{(39.2\%)} & 69.6 & \textit{(37.6\%)} & 66.1&  \textit{(37.7\%)} \\
  \hline
\end{tabular}
\end{sc}
\end{small}
\caption{Mean values of decisions by method, year and data set.  Percentages of the upper bound, Known Wind, are given for the other methods.}\label{tab:1}
\end{center}
\end{table}

\subsubsection{Wind Results}  Results are presented in Table \ref{tab:1}.  We display the value of each algorithm, along with percentages of Known Wind for the other three methods.  Function-based optimization with both types of weights outperformed the algorithm in which the state variable was ignored by a large margin ($\geq$45\% of the best possible value).  Dirichlet process weights outperformed kernel weights by a smaller but still significant margin (5.6--8.2\% of best possible value).  This is because the DP weights put substantial values on many more observations than kernel weights did in areas with high wind; in effect, kernel weights simply used the one or two closest observations in these areas for prediction.

\section{Discussion}\label{sec:discussion}  
We presented new model-free methods to solve stochastic search problems with an observable state variable, function-based optimization and gradient-based optimization.  We provided conditions for convergence for each.  Both algorithms rely on weighting observations; we gave an easily implementable weighting function (kernels) and a more complex weighting function (Dirichlet process based).  Empirical analysis shows that Dirichlet process weights add value when the state variable distribution is moderate to high dimensional or has super-Gaussian tails; this was the case in the wind commitment problem.  

More generally, this work shows that statistics and machine learning can provide solutions to currently intractable search and optimization problems.  Traditional search and optimization methods are designed to handle problems with large, complex decision spaces.  However, there are many problems where complexity is derived from elements aside from the decision, such as an observable state variable.  Statistics and machine learning offer an array of tools, such as clustering, density estimation, regression and inference, that may prove useful in solving problems that are now currently avoided.
  

\end{document}